\documentclass[a4paper, 9pt]{article}

\usepackage[all]{xy}
\usepackage{graphicx}
\usepackage{amsmath,amsthm}
\usepackage{amsfonts,amssymb}
\usepackage{float}
\usepackage{fancyhdr}
\usepackage{lscape}
\usepackage{enumerate}
\usepackage{stmaryrd}
\usepackage{pstricks}

\usepackage{hyperref}
\hypersetup{urlcolor=blue,linkcolor=black,citecolor=black,colorlinks=true}

\newtheorem{theorem}{Theorem}
\newtheorem{proposition}[theorem]{Proposition}
\newtheorem{corollary}[theorem]{Corollary}
\newtheorem{lemma}[theorem]{Lemma}

\newtheorem{definition}[theorem]{Definition}

\theoremstyle{definition}

\newtheorem{remark}[theorem]{Remark}

\newcommand{\E}{\mathbb{E}}

\newcommand{\N}{\mathbb{N}}
\newcommand{\Q}{\mathbb{Q}}
\newcommand{\R}{\mathbb{R}}
\newcommand{\Pb}{\mathbb{P}}

\newcommand{\F}{\mathcal{F}}

\newcommand{\C}{\mathcal{C}}
\newcommand{\G}{\mathcal{G}}

\newcommand{\equi}{\mathop{\sim}\limits}
\def\={{\;\mathop{=}\limits^{\text{(law)}}\;}}

\begin{document}

\begin{center}
{\LARGE \textbf{Some limiting laws associated  with the integrated \\
\vspace{.3cm}
 Brownian motion}}\\
\vspace{.7cm}
{\large Christophe \textsc{Profeta}\footnote{Laboratoire d'Analyse et Probabilit\'es, Universit\'e d'\'Evry - Val d'Essonne, B\^atiment I.B.G.B.I., 3\`eme \'etage, 23 Bd. de France, 91037 EVRY CEDEX. 
 \textbf{E-mail}: christophe.profeta@univ-evry.fr\\
 This research benefited from the support of the ``\textbf{Chaire March\'es en Mutation}'', F\'ed\'eration Bancaire Fran\c{c}aise.
 }
}\\
\vspace{.8cm}
\end{center}

\textbf{Abstract:	} We study some limit theorems for the normalized law of integrated Brownian motion perturbed by several examples of functionals: the first passage time, the $n^{\text{th}}$ passage time,  the last passage time up to a finite horizon and the supremum. We show that the penalization principle holds in all these cases and give descriptions of the conditioned processes. In particular, it is remarkable that the penalization by the $n^{\text{th}}$ passage time is independent of $n$, and always gives the same conditioned process, i.e. integrated Brownian motion conditioned not to hit 0. Our results rely on some explicit formulae obtained by Lachal and on enlargement of filtrations.\\




\section{Introduction}

The study of limiting laws or penalizations of a given process may be seen as a way to condition a probability law by an event of null probability, or by an a.s. infinite random variable. As a simple example, let for instance $(B_t, t\geq0)$ be a Brownian motion started from $x>0$, $(\F_t=\sigma(B_s, s\leq t), t\geq0)$ its natural filtration, and assume that one would like to define the law of $B$ conditioned to stay positive. Denoting the first hitting time of $B$ to level 0 by $\sigma_0=\inf\{t\geq0,\; B_t=0\}$, one natural way to do this is to look at the possible limit, as $t\rightarrow+\infty$, of the following probability family:
$$\Pb_x^{(t)}(\bullet)= \Pb_x(\bullet | T_0>t)\qquad \qquad (t\geq0).$$
In this case, for any $\Lambda_s\in \F_s$, it is easily proven that:
$$\Pb_x^{(t)}(\Lambda_s) = \frac{\E_x[1_{\Lambda_s} 1_{\{T_0>t\}} ]}{\Pb_x(T_0>t)} \xrightarrow[t\rightarrow +\infty]{} \E_x\left[ 1_{\Lambda_s} \frac{B_s}{x} 1_{\{T_0>s\}}\right].$$
Therefore, if $\Pb_x^{(\infty)}$ exists, then we must have:
$$\Pb^{(\infty)}_{x|\F_s} = \frac{B_s}{x}1_{\{T_0>s\}}  \centerdot \Pb_{x|\F_s},$$
and one recognizes here the absolute continuity formula between the Wiener measure and the law of a three-dimensional Bessel process. More generally, replacing $(1_{\{T_0>t\}}, t\geq0)$ by a general weight $(\Gamma_t, t\geq0)$, one may give the following definition of penalization:

\begin{definition}
Let $X$ be a stochastic process defined on a filtered probability space $(\Omega, \F_\infty, (\F_t)_{t\geq0}, \Pb)$ and let $(\Gamma_t,t\geq0)$ be a measurable process taking positive values, and such that $0<\E[\Gamma_t]<\infty$ for any $t>0$. We say that the process $(\Gamma_t, t\geq0)$ satisfies the penalization principle if there exists a probability measure $\Q^{(\Gamma)}$ defined on $(\Omega, \F_\infty)$ such that:
$$\forall s\geq0, \;\forall \Lambda_s\in\F_s,\qquad \lim\limits_{t\rightarrow+\infty}\frac{\E[1_{\Lambda_s} \Gamma_t]}{\E[\Gamma_t]}=\Q^{(\Gamma)}(\Lambda_s).$$
\end{definition}

The systematic study of such problems started in 2006 with a series of papers by B. Roynette, P. Vallois and M. Yor (see the survey \cite{RVY} and references within, or the monograph \cite{RY}) who look at the limiting laws of Brownian motion perturbed by different kinds of functionals: supremum, infimum, local time, length of excursion, additive functionals...  In \cite{NRY}, they managed to unify a large part of these penalizations in a general theorem where the three-dimensional Bessel process plays an important role. Some of their results were generalized to random walks \cite{Debs}, stable L\'evy processes \cite{YYY}, and linear diffusions \cite{SV, Pro1}.  We shall study here some examples of penalizations of a non-Markov process, i.e. the integrated Brownian motion
$$X_t= x +\int_0^t B_u du,$$
where $(B_u, u\geq0)$ is a standard Brownian motion. Of course, the process $(X,B)$ is Markov, with generator $\G$ given by
$$\G=\frac{1}{2} \frac{\partial^2}{\partial y^2} +y \frac{\partial }{\partial x},$$
 and we denote by $\Pb_{(x,y)}$ its law when started from $(x,y)$. We assume that $(X,B)$ is defined on the canonical space $\Omega=\C^1(\R_+\rightarrow\R) \times \C(\R_+\rightarrow\R) $ and we denote by $(\F_t=\sigma(B_s, s\leq t), \; t\geq0)$ the natural filtration generated by $B$, with $\F_\infty:= \bigvee\limits_{t\geq0}\F_t$.

\noindent
Our aim in this paper is to show that, for the integrated Brownian motion, the penalization principle holds with several kinds of functionals. We shall nevertheless see that the behavior of the penalized process is sometimes very different than the one we might have expected when comparing with classical one-dimensional Markov processes.\\

\noindent
The outline of the paper is as follows:
\begin{itemize}
\item[-] We start, in Section \ref{sec:2}, by first reviewing and further studying the penalization of integrated Brownian motion by the first hitting time of 0.
\item[-]  This result is then generalized in Section \ref{sec:3} where we study the penalization by the $n^{\text{th}}$ passage time at 0. We prove in particular that this penalization actually does not depend on $n$.
\item[-]  We next look at an intermediate approach in Section \ref{sec:4}, and consider the penalization by the last passage time at 0 before a finite horizon. In this case, we shall see that the penalized process may cross the level 0 a few times before leaving it forever.
\item[-] Section \ref{sec:5} is dedicated to a brief account on the penalization by the running supremum.
\item[-] And finally we postpone till Section \ref{sec:6} some computational and rather technical proofs.
\end{itemize}

\section{Preliminaries}\label{sec:2}

\subsection{Notations}

We start by introducing a few notation. Let $p_t(x,y;u,v)$ be the transition density of $(X, B)$:
$$p_t(x,y;u,v)= \frac{\sqrt{3}}{\pi t^2} \exp\left(-\frac{6(u-x)^2}{t^3} + \frac{6(u-x)(v+y)}{t^2} - \frac{2(v^2+vy+y^2)}{t}\right),$$
and  define the function $q_t(x,y;u,v)$ by
$$ q_t(x,y;u,v)=p_t(x,y;u,v)-p_t(x,y;u,-v).$$
Let $T_0$ be the first hitting time of 0 by the process $X$ \emph{after} time 0:
$$T_0=\inf\{t>0;\; X_t=0\}.$$
We are interested in the penalization of $X$ by the process $(1_{\{T_0>t\}}, t\geq0)$, i.e. we would like to define the law of $X$ conditioned not to hit 0. Such a study was already carried out by Groeneboom, Jongbloed and Wellner in \cite{GJW}, and we shall complete their results here.

\noindent
Define the function $h:]0,+\infty[\times\R\longmapsto\R^+$ by:
$$h(x,y)=\int_0^{+\infty}\int_0^{+\infty} w^{3/2}  q_s(x,y; 0,-w) \,ds\, dw.$$
This function is harmonic for the generator $\G$ (in the sense that $\G h=0$), and admits the following representation:
$$h(x,y)=
\begin{cases}
x^{1/6}  \left(\frac{2}{9}\right)^{1/6}  \frac{y}{x^{1/3}} U\left(\frac{1}{6}, \frac{4}{3}, \frac{2}{9} \frac{ y^3}{x}\right) & \text{if } y>0\\
\vspace{-.3cm}\\
x^{1/6} \left(\frac{2}{9}\right)^{-1/6}\frac{\Gamma(1/3)}{\Gamma(1/6)}& \text{if } y=0\\
\vspace{-.3cm}\\
- x^{1/6}  \left(\frac{2}{9}\right)^{1/6}  \frac{1}{6}   \frac{y}{x^{1/3}}  V\left(\frac{1}{6}, \frac{4}{3}, \frac{2}{9} \frac{ y^3}{x}\right) & \text{if } y<0\\
\end{cases}
$$
where $U$ denotes the confluent hypergeometric function defined on \cite[Chapter 13, p.504]{AbSt}, and $V$ is given, for $z<0$ by:
\begin{equation}\label{eq:confV}
V\left(\frac{1}{6}, \frac{4}{3}, z\right) = e^z U\left(\frac{7}{6}, \frac{4}{3},- z\right). 
\end{equation}
Since $U(a,b,z) \equi_{z\rightarrow +\infty} z^{-a}$,  we deduce by letting $x\downarrow 0$ in the expression of $h$ that:
$$h(0,y)=\lim_{x\downarrow 0}h(x,y)=\sqrt{y^+}\qquad \text{where } y^+=y\vee0=\max(y,0).$$

\noindent
The asymptotic of the survival function of $T_0$ is given in the following theorem.
\begin{theorem}[\cite{GJW}]\label{theo:GJW} For every $x>0$ and $y\in\R$, or $x=0$ and $y>0$, there is the asymptotic:
$$ \Pb_{(x,y)}(T_0>t)\;  \equi_{t\rightarrow +\infty} \;  \frac{3\Gamma(1/4)}{2^{3/4} \pi^{3/2}}  \frac{h(x,y)}{t^{1/4}}.$$
\end{theorem}
\noindent

\begin{remark}\ \\
- The rate of decay $t^{-1/4}$ is one of the few explicit persistence exponents that are known, see the recent survey on this subject by Aurzada and Simon \cite{AS}.\\
- We shall generalize this result in the next section, where we will determine the asymptotic of the $n^{\text{th}}$ passage time of $X$ at level 0.
\end{remark}
\noindent
From Theorem \ref{theo:GJW}, we deduce in particular that $x\longrightarrow h(x,y)$ and  $y\longrightarrow h(x,y)$ are increasing functions.  Observe also that $h$ admits the scaling property:
\begin{equation}\label{eq:scalingh}
h(x,y)=x^{1/6}h\left(1, \frac{y}{x^{1/3}}\right)
\end{equation}
and that there exist 2 constants $a$ and $b$ such that :
\begin{equation}\label{eq:majh}
h(x,y)\leq a x^{1/6} + b \sqrt{| y|}.
\end{equation}

\subsection{Penalization with the first passage time}

To prove penalization results, we shall, as is usual, rely on the following meta-theorem:

\begin{theorem}\label{theo:meta}
Let $(\Gamma_t, t\geq0)$ be a measurable process taking positive values and such that $0<\E[\Gamma_t]<\infty$ for every $t>0$. Assume that for any $s\geq0$:
$$ \lim\limits_{t\rightarrow+\infty}\frac{\E[\Gamma_t|\F_s]}{\E[\Gamma_t]}=M_s$$
exists a.s. and satisfies:
$$ \E[M_s]=1.$$
Then:
\begin{enumerate}[$i)$]
\item For any $s\geq 0$ and $\Lambda_s\in \F_s$:
 $$\lim\limits_{t\rightarrow+\infty}\frac{\E[1_{\Lambda_s} \Gamma_t]}{\E[\Gamma_t]}=\E[M_s 1_{\Lambda_s}]$$
that is, the convergence also holds in $L^1(\Omega)$.
\item $(M_s, s\geq0)$ is a martingale.
\item There exists a probability measure $\Q$ on $(\Omega, \F_\infty)$ such that, for any $s>0$ and $\Lambda_s\in \F_s$ :
$$\Q(\Lambda_s) =\E[M_s 1_{\Lambda_s}].$$
\end{enumerate}
\end{theorem}
\begin{proof}
The proof of this theorem is classical: Point $i)$ follows from Scheffe's lemma, Point $ii)$ is a direct consequence of Point $i)$ and Point $iii)$ follows from Kolmogorov's existence theorem since the family of probabilities $(\Pb^{(t)} := M_t \centerdot \Pb_{|\F_t}, t\geq0)$
is consistent, see \cite[Section 0.3, p.8]{RY} for a discussion on this theorem.\\
\end{proof}

\noindent
We now state the main result of this section, which is essentially a reformulation of \cite{GJW}'s result. Let  $\overline{p}_t$ denote the transition density of the process $(X,B)$ killed when $X$ hits 0:
$$\overline{p}_t(x,y;u,v)\,du\,dv = \Pb_{(x,y)}\left(X_t\in du,B_t\in dv; \;t<T_0\right).$$

\begin{theorem}\label{theo:penalGJW}
Let $x>0$ and $y\in \R$ or $x=0$ and $y>0$. 
\begin{enumerate}[$i)$]
\item For every $s\geq0$ and $\Lambda_s\in \F_s$, we have:
$$\lim_{t\rightarrow+\infty} \E_{(x,y)}\left[ \frac{1_{\{T_0>t\}}}{\Pb_{(x,y)}(T_0>t)}1_{\Lambda_s}\right]= \E_{(x,y)}\left[\frac{h(X_s,B_s )}{h(x,y)}1_{\{T_0>s\}} 1_{\Lambda_s}\right].$$
\item There exists a probability measure $\Q_{(x,y)}$ defined on $(\Omega, \F_\infty)$ such that
$$\forall s\geq0,\; \forall \Lambda_s\in \F_s,\qquad \Q_{(x,y)}(\Lambda_s) =  \E_{(x,y)}\left[\frac{h(X_s,B_s )}{h(x,y)}1_{\{T_0>s\}} 1_{\Lambda_s}\right].$$
\item Under $\Q_{(x,y)}$, the process $(X_t, t\geq0)$ is a.s. transient and never hits 0:
$$\Q_{(x,y)}\left(\lim_{t\rightarrow +\infty}X_t =+\infty \right)=1 \qquad \text{ and }\qquad \Q_{(x,y)}(T_0=+\infty)=1.$$
\item The coordinate process $(X,B)$ under $\Q_{(x,y)}$ has the same law as  the solution of the system of SDEs:
\begin{equation}\label{eq:system}
\begin{cases}
dX_t = B_t dt&\qquad \quad X_0=x\\
\displaystyle dB_t= dW_t +  \frac{1}{h(X_t, B_t)}\frac{\partial h}{\partial y}(X_t, B_t)dt&\qquad\quad B_0=y
\end{cases}
\end{equation}
where $(W_t, t\geq0)$ is a $\Q_{(x,y)}$-Brownian motion. Its transition density is given by:
$$\Q_{(x,y)}\left(X_t\in du,B_t\in dv\right)= \frac{1}{h(x,y)} \overline{p}_t(x,y;u,v)h(u,v)\,du\,dv .$$
\end{enumerate}
\end{theorem}

\begin{proof}
Applying the Markov property:
$$\Pb_{(x,y)}\left(T_0>t|\F_s\right) = 1_{\{T_0>s\}} \Pb_{(X_s, B_s)}(T_0>t-s) \;\equi_{t\rightarrow +\infty}  \;\frac{3\Gamma(1/4)}{2^{3/4} \pi^{3/2}}  \frac{h(X_s,B_s)}{t^{1/4}} 1_{\{T_0>s\}}$$
hence, we have the a.s. convergence:
$$\lim_{t\rightarrow+\infty} \frac{\Pb_{(x,y)}\left(T_0>t|\F_s\right)}{\Pb_{(x,y)}(T_0>t)}=\frac{h(X_s,B_s )}{h(x,y)}1_{\{T_0>s\}}$$
and the convergence in $L^1(\Omega)$ as well as the existence of $\Q_{(x,y)}$ will follow from Theorem \ref{theo:meta} once we have proven that
$$\forall s\geq0, \qquad  \E_{(x,y)}\left[h(X_s,B_s )1_{\{s<T_0\}}\right] =h(x,y).$$
But it is known from \cite{GJW} that the function 
$$(u,v)\longrightarrow \frac{1}{h(x,y)} \overline{p}_t(x,y;u,v) h(u,v)$$
is a probability density on $[0,+\infty[\times \R$. Therefore,
$$
\E_{(x,y)}\left[h(X_t, B_t)1_{\{t<T_0\}}\right]=\int_0^{+\infty} \int_{\R} h(u,v) \overline{p}_t(x,y;u,v)\, dv\, du=h(x,y)
$$
which is the desired result. Next, the fact that $X$ a.s. never hits 0 is immediate since:
$$\Q_{(x,y)}(T_0=\infty)= \lim_{t\rightarrow +\infty}\Q_{(x,y)}(T_0>t)= \lim_{t\rightarrow +\infty}\frac{\E_{(x,y)}[h(X_t, B_t)1_{\{T_0>t\}}]}{h(x,y)}= \lim_{t\rightarrow +\infty}1 =1.$$
To prove that $X$ is transient, observe that the process $\displaystyle \left(N_t = \frac{1}{h(X_t, B_t)}, \; t\geq0\right)$ is a positive $\Q_{(x,y)}$-local martingale which therefore converges $\Q_{(x,y)}$-a.s. toward a r.v. $N_\infty$. But, as $\Q_{(x,y)}\left[N_t\right] = \frac{1}{h(x,y)}\Pb_{(x,y)}(T_0>t) \xrightarrow[t\rightarrow +\infty]{} 0$, we deduce that $N_\infty=0$, which, given the behavior of $h$ at infinity, implies that $X_t \xrightarrow[t\rightarrow +\infty]{} +\infty$  $\Q_{(x,y)}$-a.s.\\
Finally, to study the law of $(X,B)$ under $\Q_{(x,y)}$, we shall apply Girsanov's theorem. Indeed, on the set $\{t<T_0\}$, It\^o's formula yields, since $h$ is harmonic for $\G$ :
$$\ln(h(X_t, B_t)) = \int_0^{t} \frac{1}{h(X_s, B_s)}\frac{\partial h}{\partial y}(X_s, B_s)dB_s -  \frac{1}{2}\int_0^{t} \frac{1}{h^2(X_s, B_s)}\left(\frac{\partial h}{\partial y}(X_s, B_s)\right)^2ds$$
and, since the process
\begin{align*}
Z_{t\wedge T_0}&=\exp\left( \int_0^{t\wedge T_0} \frac{1}{h(X_s, B_s)}\frac{\partial h}{\partial y}(X_s, B_s)dB_s -  \int_0^{t\wedge T_0} \frac{1}{h^2(X_s, B_s)}\left(\frac{\partial h}{\partial y}(X_s, B_s)\right)^2ds\right)\\
&=h(X_t, B_t)1_{\{t<T_0\}}   \qquad (\text{since }\Pb_{(x,y)}(B_{T_0}<0)=1)
\end{align*}
is a true martingale, Girsanov's theorem implies that the process $(X, B)$ under $\Q_{(x,y)}$ is a weak solution of the System (\ref{eq:system}) on the set $\{t<T_0\}$. But as $\Q_{(x,y)}(T_0=+\infty)=1$, we conclude that the coordinate process $(X, B)$ under $\Q_{(x,y)}$ is a solution of (\ref{eq:system}) on $[0,+\infty[$.\\
\end{proof}

\begin{remark}
The direct proof of the existence and uniqueness of the solution of the System (\ref{eq:system}) is not straightforward since the coefficients do not satisfy the usual Lipchitz and growth conditions. Nevertheless, it is proven in \cite{GJW} by a localization argument that (\ref{eq:system}) admits a unique strong solution for $x>0$ and $y\in \R$. If $x=0$ and $y>0$, we may also construct a unique strong solution as follows: define $\sigma_ \varepsilon=\inf\{t\geq0,\, B_t=\varepsilon\}$ and assume that $\varepsilon$ is such that $y>\varepsilon$. The function 
$$(u,v)\longmapsto  \frac{1}{h(u,v\vee\varepsilon)}\frac{\partial h}{\partial v}(u,v\vee \varepsilon)=\frac{1}{y\vee\varepsilon}-\frac{1}{9}\frac{(y\vee\varepsilon)^2}{x} \frac{U\left(\frac{7}{6}, \frac{7}{3}, \frac{2}{9} \frac{ (y\vee\varepsilon)^3}{x}\right) }{U\left(\frac{1}{6}, \frac{4}{3}, \frac{2}{9} \frac{ (y\vee\varepsilon)^3}{x}\right) }$$
 is globally Lipschitz, so we may construct a unique strong solution of  (\ref{eq:system}) up to time $\sigma_ \varepsilon$. By pasting this solution with the solution of the system started from $(X_{\sigma_ \varepsilon},  \varepsilon)$ (which is known to exist since $X_{\sigma_ \varepsilon}>0$ a.s.), we obtain a solution of  (\ref{eq:system})  on $[0,+\infty[$. Furthermore, from \cite[Theorem 4.20, p.322]{KS}, since the coefficients are bounded on compact subsets, this solution is strongly Markovian.
\end{remark}

\subsection{Passage times of integrated Brownian motion conditioned to be positive}

As mentioned earlier, the solution of the system (\ref{eq:system}) was studied in \cite{GJW}, where the authors prove in particular that the first component $X$ is transient and goes towards $+\infty$ $\Q_{(x,y)}$-a.s. We shall complete their results by computing some first and last passage times, thanks to the weak absolute continuity formula. Let $T_a=\inf\{t>0,\; X_t=a\}$ and $\sigma_b=\inf\{t\geq0,\; B_t=b\}$.

\begin{theorem}
Let $x>0$ and $y\in \R$. 
\begin{enumerate}[$i)$]
\item  For $a< x$,  the density of the couple $(T_a, B_{T_a})$ under $\Q_{(x,y)}$ is given by
$$
\Q_{(x,y)}(T_a\in dt, B_{T_a}\in dz)=\frac{h(a,z)}{h(x,y)}\Pb_{(x,y)}(T_a\in dt, B_{T_a}\in dz).
$$
\item For $y\geq b\geq0$, the density of the couple $(\sigma_b, X_{\sigma_b})$ under $\Q_{(x,y)}$ is given by:
$$\Q_{(x,y)}(\sigma_b \in dt,\; X_{\sigma_b} \in dz )=\frac{h(z,b)}{h(x,y)}\Pb_{(x,y)}(\sigma_b \in dt,\; X_{\sigma_b} \in dz ).  $$

\end{enumerate}
\end{theorem}

\begin{proof}
The proof is straightforward and relies on Doob's stopping theorem:
$$
\Q_{(x,y)}(T_a>t, B_{T_a}<z)=\frac{1}{h(x,y)}\E_{(x,y)}\left[h(a, B_{T_a})1_{\{T_a>t, B_{T_a}<z\}}\right]
$$
as, for $x>a$, the continuity of paths implies that $T_a<T_0$. Similarly, for $y\geq  b\geq 0$, we must have $\sigma_b < T_0$:
$$
\Q_{(x,y)}(\sigma_b>t, X_{\sigma_b}<z)=\frac{1}{h(x,y)}\E_{(x,y)}\left[h(X_{\sigma_b}, b)1_{\{\sigma_b>t, X_{\sigma_b}<z\}}\right].
$$
\end{proof}

\begin{remark}
The same proof applies when $a=x$ and $y>0$ since in this case, the process $(X_t, t\geq0)$ is strictly increasing in the neighborhood of $X_0=x$. This leads to a simple expression:
\begin{multline*}
\Q_{(a,y)}(T_a\in dt,  B_{T_a}\in dz)\\
= \frac{h(a,z)}{h(a,y)} \frac{3|z|}{t^2 \pi\sqrt{2}}\exp\left(-\frac{2}{t}(y^2 - y|z| +z^2)\right) \left(\int_0^{4y|z|/s} \frac{1}{\sqrt{\pi\theta}} e^{-\frac{3}{2}\theta}d\theta\right)\; 1_{]-\infty,0]}(z) dt dz.
\end{multline*}
\end{remark}

\begin{corollary}\label{cor:Tafini}
Let $x>a\geq0$:
$$\Q_{(x,y)}(T_a<+\infty) = 1-\frac{h(x-a, y)}{h(x,y)}.$$
\end{corollary}

\noindent
The proof of this corollary is a direct consequence of the following lemma:

\begin{lemma}\label{lem:hBTa}
Let $a<x$. Then:
$$ \E_{(x,y)}[h(a, B_{T_a})]=h(x,y)-h(x-a,y)$$
\end{lemma}
\begin{proof}
It does not seem easy to compute directly this expression using the explicit distribution of $B_{T_a}$ (see Lachal \cite{Lac1}), so we shall rather rely on a martingale argument. From Doob's stopping theorem, since $T_a<T_0$ a.s.:
$$h(x,y)=\E_{(x,y)}[h(X_{t\wedge T_a},B_{t\wedge T_a}) 1_{\{t\wedge T_a<T_0\}}] =\E_{(x,y)}[h(a,B_{T_a}) 1_{\{T_a<t\}}]+
  \E_{(x,y)}[h(X_{t},B_t) 1_{\{t<T_a\}}]$$
hence, passing to the limit:
$$ \E_{(x,y)}[h(a, B_{T_a})]=h(x,y) - \lim_{t\rightarrow +\infty } \E_{(x,y)}[h(X_{t},B_t) 1_{\{t<T_a\}}].$$
Now, by translation,
$$\E_{(x,y)}[h(X_{t},B_t) 1_{\{t<T_a\}}] =\E_{(x-a,y)}[h(a+X_{t},B_t) 1_{\{t<T_0\}}]. $$
Then recall that the functions $x\rightarrow h(x,y)$ and  $y\rightarrow h(x,y)$ are increasing,
so we need to study:
\begin{align*}
\big|\E_{(x,y)}[h(X_{t},B_t) 1_{\{t<T_a\}}]  - h(x-a,y)\big|
&= \E_{(x-a,y)}[h(a+X_{t},B_t) 1_{\{t<T_0\}}] -  \E_{(x-a,y)}[h(X_{t},B_t) 1_{\{t<T_0\}}]   \\
&=\int_0^{+\infty} \int_\R  \left(h\left(a+u,v \right)  - h\left(u,v\right)  \right)   \overline{p}_t(x-a,y; u,  v)  du dv.
\end{align*}
Let $A >0$ be fixed. We first write:
$$\int_0^{A} \int_\R  \left(h\left(a+u,v \right)  - h\left(u,v\right)  \right)   \overline{p}_t(x-a,y; u,  v)  du dv \leq 
\int_0^{A} \int_\R  \left(h\left(a+A,v \right)  - h\left(0,v\right)  \right)   \overline{p}_t(x-a,y; u,  v)  du dv .$$
Now, observe that the function 
$$v\longrightarrow h\left(a+A,v \right)  - h\left(0,v\right)$$
is bounded by a constant, say $K$. Indeed, it is a positive and continuous function whose limits at $+\infty$ and $-\infty$ both equal 0.
Therefore,
\begin{align*}
\int_0^{A} \int_\R  \left(h\left(a+A,v \right)  - h\left(0,v\right)  \right)   \overline{p}_t(x-a,y; u,  v)  du dv 
&\leq K \int_0^{A} \int_\R  \overline{p}_t(x-a,y; u,  v)  du dv \\
&\leq K\, \Pb_{(x-a,y)} (T_0>t) \xrightarrow[t\rightarrow +\infty]{} 0.
\end{align*}
Now, adding and sustracting  $(a+u) ^{1/6} h\left(1, \frac{v}{u^{1/3}}\right)$, we decompose the remaining integral in two terms:
$$\int_A^{+\infty} \int_\R  \left((a+u) ^{1/6} h\left(1, \frac{v}{(a+u)^{1/3}}\right)  - u^{1/6} h\left(1, \frac{v}{u^{1/3}}\right) \right)   \overline{p}_t(x-a,y; u,  v)  du dv = I_t^{(1)}+I_t^{(2)}$$
with
$$
\begin{cases}
I_t^{(1)}  &= \displaystyle  \int_A^{+\infty} \int_\R (a+u) ^{1/6} \left(h\left(1, \frac{v}{(a+u)^{1/3}}\right)  - h\left(1, \frac{v}{u^{1/3}}\right) \right)\overline{p}_t(x-a,y; u,  v) du dv,\\
\vspace{-.3cm}\\
I_t^{(2)}&=\displaystyle\int_A^{+\infty} \int_\R \left((a+u) ^{1/6}  -u^{1/6}\right) h\left(1, \frac{v}{u^{1/3}}\right)  \overline{p}_t(x-a,y; u,  v)  du dv. 
\end{cases}
$$
\begin{enumerate}[$i)$]
\item The second term may be dealt with easily :
\begin{align*}
I_t^{(2)}&=  \int_A^{+\infty} \int_\R \left(\left(\frac{a}{u}+1\right)^{1/6}-1\right) h\left(u, v\right)   \overline{p}_t(x-a,y; u,  v)  du dv\\
&=h(x-a, y)\Q_{(x-a,y)}\left[ \left(\left(\frac{a}{X_t}+1\right)^{1/6}-1\right) 1_{\{X_t\geq A\}}\right]\xrightarrow[t\rightarrow +\infty]{} 0
\end{align*}
from the dominated convergence theorem, since $\displaystyle \Q_{(x-a,y)}\left(\lim_{t\rightarrow +\infty} X_t =+\infty\right)=1$.
\item As for the first term, we write, since $y\longmapsto h(1,y)$ is increasing:
\begin{align}
\notag I_t^{(1)}&\leq \int_A^{+\infty} \int_{-\infty}^0 (a+u) ^{1/6} \left(h\left(1, \frac{v}{(a+u)^{1/3}}\right)  - h\left(1, \frac{v}{u^{1/3}}\right) \right)\overline{p}_t(x-a,y; u,  v)  du dv \\ 
\notag &= \int_A^{+\infty} \int_{-\infty}^0 (a+u) ^{1/6} \left(h\left(1, \frac{v}{u^{1/3}}\frac{u^{1/3}}{(a+u)^{1/3}}\right)  - h\left(1, \frac{v}{u^{1/3}}\right) \right)\overline{p}_t(x-a,y; u,  v)  du dv \\
\label{eq:I1} &\leq \int_A^{+\infty} \int_{-\infty}^0 (a+u) ^{1/6} \left(h\left(1, \frac{v}{u^{1/3}}\frac{1}{(\frac{a}{A}+1)^{1/3}}\right)  - h\left(1, \frac{v}{u^{1/3}}\right) \right)\overline{p}_t(x-a,y; u,  v)  du dv.
\end{align}
Let $\varepsilon>0$. Since  $\displaystyle \lim_{z\rightarrow -\infty} h(1,z)=0$, there exists $-b<0$ such that
$$\forall z\leq -b, \qquad h\left(1,\frac{z}{2}\right) \leq \frac{\varepsilon}{2}.$$
In particular, if $A$ is  such that
$$\frac{1}{(\frac{a}{A}+1)^{1/3}} \geq \frac{1}{2},$$
then, for $z\leq -b$:
$$\left|h\left(1,\frac{z}{(\frac{a}{A}+1)^{1/3}}\right)-h(1,z)\right| \leq h\left(1,\frac{z}{2}\right) + h\left(1,z\right) \leq \varepsilon.$$
Next, from Heine's theorem, the function $z\longrightarrow h(1, z)$ is uniformly continuous in the compact set $[-b,0]$.  Therefore, there exists $\eta>0$ such that
$$\forall (y,z) \in [-b,0]^2,\qquad |z-y |\leq \eta \Longrightarrow  |h(1,z)-h(1,y) | \leq \varepsilon.$$ 
Thus, if we take $\displaystyle y=\frac{z}{(\frac{a}{A}+1)^{1/3}}$ and assume that $A$ is chosen large enough such that
$$|b| \left(1-\frac{1}{(\frac{a}{A}+1)^{1/3}}\right) \leq \eta,$$
we deduce that, for every $z\in [-b,0]$, we also have:
$$\left|h\left(1,\frac{z}{(\frac{a}{A}+1)^{1/3}}\right)-h(1,z)\right|\leq \varepsilon.$$
Now, going back to (\ref{eq:I1}), we may write:
\begin{align*}
I_t^{(1)}&\leq  \varepsilon \int_A^{+\infty} \int_{-\infty}^0 (a+u) ^{1/6} \overline{p}_t(x-a,y; u,  v)  du dv\\
&\leq \varepsilon  \left(\frac{a}{A}+1\right)^{1/6} \int_A^{+\infty} \int_{-\infty}^0 u^{1/6} \overline{p}_t(x-a,y; u,  v)  du dv\\
&\leq \varepsilon \left(\frac{a}{A}+1\right)^{1/6} \; \E_{(x-a,y)}\left[X_t^{1/6} 1_{\{t<T_0\}} 1_{\{B_t\leq 0\}}  \right]\\
&= \varepsilon \left(\frac{a}{A}+1\right)^{1/6}\; \E_{(x-a,y)}\left[X_t^{1/6} 1_{\{t<T_0\}} |B_t\leq 0  \right]  \Pb_{(x,y)}(B_t\leq 0)
\end{align*}
Observe next that, by decomposing
$$X_t=x+\int_0^{\gamma_0^{(t)}} B_s ds + \int_{\gamma_0^{(t)}}^t B_s ds\quad \text{ and }\quad 
\{t<T_0\} = \left\{\inf_{u\leq t} x +\int_0^{\gamma_0^{(t)}\wedge u} B_s ds + \int_{\gamma_0^{(t)}\wedge u}^{u} B_s ds >0\right\}$$
with $\gamma_0^{(t)}=\sup\{s\leq t,\; B_s=0\}$, we deduce that:
\begin{align*}
I_t^{(1)}&\leq \varepsilon \left(\frac{a}{A}+1\right)^{1/6}\; \E_{(x-a,y)}\left[X_t^{1/6} 1_{\{t<T_0\}} |B_t\geq 0  \right]  \Pb_{(x,y)}(B_t\leq 0)\\
&=  \varepsilon\left(\frac{a}{A}+1\right)^{1/6} \; \E_{(x-a,y)}\left[X_t^{1/6} 1_{\{t<T_0\}} 1_{\{B_t\geq 0\}}   \right]  \frac{\Pb_{(x,y)}(B_t\leq 0)}{\Pb_{(x,y)}(B_t\geq 0)}\\
&\leq\varepsilon  \left(\frac{a}{A}+1\right)^{1/6} \; \E_{(x-a,y)}\left[\frac{h(X_t, B_t)}{h(1,0)} 1_{\{t<T_0\}} 1_{\{B_t\geq 0\}}   \right]  \frac{\Pb_{(x,y)}(B_t\leq 0)}{\Pb_{(x,y)}(B_t\geq 0)}\\
&\leq \varepsilon  \left(\frac{a}{A}+1\right)^{1/6} \; \frac{h(x-a,y)}{h(1,0)} \frac{\Pb_{(x,y)}(B_t\leq 0)}{\Pb_{(x,y)}(B_t\geq 0)} \\
&\xrightarrow[t\rightarrow +\infty]{} \varepsilon\left(\frac{a}{A}+1\right)^{1/6} \; \frac{h(x-a,y)}{h(1,0)},
\end{align*}
where we have used that:
$$x^{1/6} = \frac{h(x,0)}{h(1,0)} \leq  \frac{h(x,y)}{h(1,0)} \qquad \text{for } y\geq0.$$
\end{enumerate}
\end{proof}



\begin{theorem}
Let $x>0$ and $y\in \R$ or $x=0$ and $y>0$. 
The cumulative distribution function of the last passage time of $X$ at level $a$ under $\Q_{(x,y)}$ is given by:
$$\Q_{(x,y)}(g_a< t) = \frac{1}{h(x,y)} \int_\R \int_0^{+\infty} h(u,v) \overline{p}_t(x,y; u+a,v) du dv$$
with $\displaystyle g_a=\sup\{t\geq0,\; X_t=a\}$.
\end{theorem}

\begin{proof} Using the Markov property and Corollary \ref{cor:Tafini} :
\begin{align*}
\Q_{(x,y)}(g_a<t)&=\Q_{(x,y)}(g_a<t  \cap X_t> a)\\
&=\int_\R \int_a^{+\infty}  \Q_{(x,y)}(g_a<t | X_t=u, B_t=v )  \Q_{(x,y)}(X_t\in du, B_t \in dv)\\
&=\frac{1}{h(x,y)}\int_\R \int_a^{+\infty}  \Q_{(x,y)}(T_a\circ\theta_t=+\infty | X_t=u, B_t=v )\overline{p}_t(x,y;u,v)h(u,v)du dv\\
&=\frac{1}{h(x,y)}\int_\R \int_a^{+\infty}  \Q_{(u,v)}(T_a=+\infty) \overline{p}_t(x,y;u,v)h(u,v)du dv\\
&=\frac{1}{h(x,y)}\int_\R \int_a^{+\infty}  h(u-a,y) \overline{p}_t(x,y;u,v)du dv\\
&=\frac{1}{h(x,y)}\int_\R \int_0^{+\infty}  h(u,y) \overline{p}_t(x,y;u+a,v)du dv.
\end{align*}
\end{proof}

\section{Penalization with the $n^\text{th}$ passage time}\label{sec:3}

Let $T_0^{(n)}$ be the  $n^\text{th}$ passage time of $X$ at the level 0:
$$T_0^{(0)}=0 \qquad \text{and}\qquad T_0^{(n)}=\inf\{t>T_0^{(n-1)};\; X_t=0\}.$$
We are now interested in the penalization of $X$ by the process $(1_{\{T_0^{(n)}>t\}}, t\geq0)$. To this end, we first need to compute some asymptotics:
\begin{proposition}\label{prop:T0}
Let $b\neq 0$. There is the asymptotic, for $n\geq1$:
$$ \Pb_{(0,b)}(T_0^{(n)}>t)\; \equi_{t\rightarrow +\infty}\;    \frac{ 2^{1/4}  \Gamma(1/4) \sqrt{|b|}}{\sqrt{\pi}(n-1)!}
 \left(\frac{9}{4\pi^2}\right)^{n/2}  \frac{\left(\ln(t)\right)^{n-1} }{ t^{1/4}} . $$
\end{proposition}

\noindent
The proof of this result is given in Section \ref{sec:6} as it is purely computational and rather technical.

\begin{remark}
The asymptotic behavior in $n$ was already studied by McKean in \cite{McK} (with a slightly different definition of $T_0^{(n)}$) where he derives the following results:
$$\ln(T_0^{(n)})\; \equi_{n\rightarrow+\infty}\;  \frac{8\pi}{\sqrt{3}}n \qquad \text{and} \qquad \ln\left|B_{T_0^{(n)}}\right|\; \equi_{n\rightarrow+\infty}\;  \frac{4\pi}{\sqrt{3}}n.$$
\end{remark}

\begin{theorem}
Let $x>0$ and $y\in \R$. 
\begin{enumerate}[$i)$]
\item There is the asymptotic:
$$ \Pb_{(x,y)}(T_0^{(n)}>t)\; \equi_{t\rightarrow +\infty} \;   \frac{ 2^{1/4}  \Gamma(1/4)}{\sqrt{\pi}(n-1)!}
 \left(\frac{9}{4\pi^2}\right)^{n/2}  \frac{\left(\ln(t)\right)^{n-1} }{ t^{1/4}} h(x,y) . $$
\item For every $s\geq0$ and $\Lambda_s\in \F_s$, we have:
$$\lim_{t\rightarrow+\infty} \E_{(x,y)}\left[ \frac{1_{\Lambda_s}1_{\{T_0^{(n)}>t\}}}{\Pb_{(x,y)}(T_0^{(n)}>t)}\right]= \E_{(x,y)}\left[\frac{h(X_s,B_s )}{h(x,y)}1_{\{T_0>s\}} 1_{\Lambda_s}\right].$$
Therefore, the penalization by the $n^{\text{th}}$ passage time leads to the same conditioned process as the penalization by the first passage time.
\end{enumerate}
\end{theorem}

\begin{remark}
This behavior is very different from the classical one-dimensional Markov processes which have already been studied in the literature. Indeed, in all the known examples (random walks, Brownian motion, recurrent diffusions...) the penalization by the $n^{\text{th}}$ passage time (or by its analogue, the inverse local time) gives penalized processes which highly depend on $n$, see \cite{Debs, RVY3, SV}...
\end{remark}

\begin{proof}
The proof of Point $i)$ follows from Proposition \ref{prop:T0} and from the Markov property. Indeed, from the formula (see \cite{Lac1} and \cite{GJW})
$$ \Pb_{(x,y)}(T_0\in ds, B_{T_0}\in dz) = |z| \left(q_s(x,y;0,z)-\int_0^s  \int_0^{+\infty}q_{s-u}(x,y;0,w) \Pb_{(0,z)}(T_0\in du, B_{T_0}\in dw)  \right) $$
we may write:
\begin{align*}
&\Pb_{(x,y)}(T_0^{(n)}>t)-\Pb_{(x,y)}(T_0>t)\\
& =\int_0^t  \int_{-\infty}^0 \Pb_{(0,z)}(T_0^{(n-1)}>t-s) \Pb_{(x,y)}(T_0\in ds, B_{T_0}\in dz)\\
&=\int_0^t  ds\int_{-\infty}^0dz \Pb_{(0,z)}(T_0^{(n-1)}>t-s) |z| \left(q_s(x,y;0,z)-\int_0^s  \int_0^{+\infty}q_{s-u}(x,y;0,w) \Pb_{(0,z)}(T_0\in du, B_{T_0}\in dw)  \right) 
\end{align*}
We now take the Laplace transform of both sides:
\begin{align*}
&\int_0^{+\infty} e^{-\lambda t}  \Pb_{(x,y)}(T_0^{(n)}>t) dt -\int_0^{+\infty} e^{-\lambda t}  \Pb_{(x,y)}(T_0>t) dt \\
&= \int_{-\infty}^0 \int_0^{+\infty }\left(\int_0^{+\infty} e^{-\lambda t}  \Pb_{(0,z)}(T_0^{(n-1)}>t) dt \right) \left(\int_0^{+\infty}e^{-\lambda t}q_t(x,y;0,w)dt \right)  \E_{(0,z)}\left[e^{-\lambda T_0}, B_{T_0}\in dw\right] zdz\\
&\quad-
\int_{-\infty}^0 \left(\int_0^{+\infty} e^{-\lambda t}  \Pb_{(0,z)}(T_0^{(n-1)}>t) dt \right) \left(\int_0^{+\infty}e^{-\lambda t}q_t(x,y;0,z)dt  \right) zdz.
\end{align*}
But, from the symmetry relation:
$$\E_{(0,z)}\left[e^{-\lambda T_0}, B_{T_0}\in dw\right]dz = - \frac{w}{z}\E_{(0,w)}\left[e^{-\lambda T_0}, B_{T_0}\in dz\right] dw\qquad\qquad  (wz<0)$$
and the fact that $q_t(x,y; u, v)= -q_t(x,y; u, -v)$, we deduce that:
\begin{align*}
&\int_0^{+\infty} e^{-\lambda t}  \Pb_{(x,y)}(T_0^{(n)}>t) dt -\int_0^{+\infty} e^{-\lambda t}  \Pb_{(x,y)}(T_0>t) dt \\
&\quad=\int_{0}^{+\infty}\left(\int_0^{+\infty} e^{-\lambda t}  (\Pb_{(0,w)}(T_0^{(n)}>t) -\Pb_{(0,w)}(T_0>t))dt   \right) \left(\int_0^{+\infty}e^{-\lambda t}q_t(x,y;0,-w)dt \right)wdw\\
&\quad -\int_{-\infty}^0 \left(\int_0^{+\infty} e^{-\lambda t}  \Pb_{(0,z)}(T_0^{(n-1)}>t) dt \right) \left(\int_0^{+\infty}e^{-\lambda t}q_t(x,y;0,z)dt  \right) zdz
\end{align*}
and the result follows from the Tauberian theorem.\\
\noindent
To prove Point $ii)$, we finally write, applying the Markov property:
\begin{multline*}
\Pb_{(x,y)}(T_0^{(n)}>t |\F_s) = 1_{\{T_0>s\}}\Pb_{(X_s,B_s)}(T_0^{(n)}>t-s) + 1_{\{T_0\leq s, T_0^{(2)}>s\}}\Pb_{(X_s,B_s)}(T_0^{(n-1)}>t-s) \\
+ \ldots + 1_{\{T_0^{(n-1)}\leq s,T_0^{(n)}>s\} }\Pb_{(X_s,B_s)}(T_0>t-s),
\end{multline*}
and, due to Point $i)$, the leading asymptotic comes from the first term  
$1_{\{T_0>s\}}\Pb_{(X_s,B_s)}(T_0^{(n)}>t-s)$.\\
\end{proof}

\section{Penalization with the last passage time up to a finite horizon}\label{sec:4}

We have seen that the penalization with the $n^{\text{th}}$ passage time gives a process which never hit 0. We shall try to obtain an intermediate penalization by choosing as weight process a function of the last passage time up to a finite horizon.\\

We define :  
$$g_0^{(t)}=\sup\{u\leq  t;\;  X_u=0\}.$$

\begin{lemma}\label{lem:g}
The density of the triplet $(g_0^{(t)}, X_t, B_t)$ is given by :
\begin{equation}
\Pb_{(x,y)}\left(g_0^{(t)}\in ds,\; X_t\in du,\; B_t\in dv  \right)=  \left(\int_\R |z|  p_{s}(0, z; -x, y) \overline{p}_{t-s}(0, z; u, v ) dz\right)  \,ds  dudv
\end{equation}
\end{lemma}

\begin{proof}
By a time reversal argument (see \cite[Lemma 2.12]{LacExc}), we have :
$$\Pb_{(x,y)}\left(g_0^{(t)}> s,\; X_t\in du,\; B_t\in dv  \right)/(dudv)=\Pb_{(u,v)}^\ast\left(T_0<t-s,\; X_t\in dx,\; B_t\in dy  \right)/(dxdy)$$
where $\Pb^\ast$ denotes the law of the dual process of $(X,B)$, whose distribution equals that of $(X, -B)$. In particular, we have
$$  p_{s}^\ast\left(0,v; x, y\right)=p_s(0,v; -x, y)$$
and
$$\Pb_{(u,v)}^\ast\left(T_0\in ds, B_{T_0}\in dz  \right) = |z| \overline{p}_s(0, z; u, v ) du dz.$$
Then, applying the Markov property:
\begin{align*}
\Pb_{(u,v)}^\ast\left(T_0<t-s,\; X_t\in dx,\; B_t\in dy  \right)/(dxdy)&=\E_{(u,v)}^\ast\left[1_{\{T_0<t-s\}} p_{t-T_0}^\ast\left(0,B_{T_0}\; x, y\right)\right]\\
&=\E_{(u,v)}^\ast\left[1_{\{T_0<t-s\}}  p_{t-T_0}(0, B_{T_0}, -x, y) \right] \\
&=\int_0^{t-s} \int_\R   p_{t-r}(0, z; -x, y) |z| \overline{p}_r(0, z; u, v ) dr dz
\end{align*}
and the result follows by differentiation with respect to $s$.
\end{proof}

\noindent
We  set
$$\Phi(x,y)=\varphi(0)h(x,y)1_{\{x\geq 0\}}   + \varphi(0)h(-x, -y)1_{\{x\leq 0\}} +  \int_\R |z|^{3/2} \int_0^{+\infty} \varphi(s) p_s(x,y;0,z) dsdz.$$
Note that when $x=0$, one of the two symmetric terms is always null. Then we may state the following theorem:
\begin{theorem}
Let  $\varphi:\R^+\longrightarrow \R^+$ be a continuous function with compact support.
\begin{enumerate}[$i)$]
\item The process
\begin{multline*}
M_t^\varphi=  \varphi(g_0^{(t)}) \left(h(X_t,B_t)1_{\{X_t\geq  0\}}+ h(-X_t, -B_t)1_{\{X_t\leq 0\}}\right) \\
+ \int_\R|z|^{3/2} \int_0^{+\infty} \varphi(t+s) p_s(X_t,B_t;0,z) dsdz
\end{multline*}
is a positive martingale which converges toward 0 as $t\rightarrow +\infty$.
\item Let $s\geq0$ and $(x,y) \in  \R^2$. For any $\Lambda_s \in \F_s$, we have:
$$\lim_{t\rightarrow +\infty}\frac{\E_{(x,y)}\left[ 1_{\Lambda_s}  \varphi(g_0^{(t)}) \right]}{\E_{(x,y)}\left[\varphi(g_0^{(t)}) \right]} =  \E_{(x,y)}\left[ 1_{\Lambda_s}  \frac{M_s^\varphi}{\Phi(x,y)} \right].$$
\item There exists a family of probabilities $(\Q^\varphi_{(x,y)}, \; (x,y)\in \R^2)$ such that, for any $t\geq0$:
$$\Q^\varphi_{(x,y)|\F_t} = \frac{M_t^\varphi}{\Phi(x,y)} \centerdot \Pb_{(x,y)|\F_t}.$$
\item Let $g_0=\sup\{u\geq0;\;  X_u=0\}$. Then, $\Q^\varphi_{(x,y)}(g_0<+\infty)=1$ and conditionally on $g_0$ and $B_{g_0}$:
\begin{enumerate}[$i)$]
\item the processes $(X_u, u\leq g_0)$ and $(X_{u+g_0},u\geq0)$ are independent,
\item the process $(X_{u+g_0},u\geq0)$ has the same law as integrated Brownian motion started from $(0, B_{g_0})$ and conditioned to stay positive if $B_{g_0}>0$, or conditioned to stay negative if $B_{g_0}<0$.
\end{enumerate}
\end{enumerate}
\end{theorem}

\begin{proof}
From Lemma \ref{lem:g}, we have:
$$\E_{(x,y)}[\varphi(g_0^{(t)})] = \varphi(0) \Pb_{(x,y)}(T_0>t) + \int_0^t ds\, \varphi(s)  \int_\R |z| p_s(x,y; 0,z) \Pb_{(0,z)}(T_0>t-s) dz $$ 
and the asymptotic
\begin{align*}
\E_{(x,y)}[\varphi(g_0^{(t)})] & \;\equi_{t\rightarrow+\infty} \; \frac{3\Gamma(1/4)}{2^{3/4} \pi^{3/2} t^{1/4}}\bigg( \varphi(0)h(x,y)1_{\{x\geq 0\}}   + \varphi(0)h(-x, -y)1_{\{x\leq 0\}}\\
& \hspace{5cm}\left. +  \int_\R  |z|h(0,|z|) \int_0^{+\infty} \varphi(s) p_s(x,y;0,z) dsdz\right)\\
& \;\equi_{t\rightarrow+\infty}  \;\frac{3\Gamma(1/4)}{2^{3/4} \pi^{3/2} t^{1/4}}\Phi(x,y).
\end{align*}
 Applying the Markov property:
$$
\E_{(x,y)}\left[\varphi(g_0^{(t)}) | \F_s\right] =\varphi(g_0^{(s)}) \Pb_{(X_s, B_s)}(T_0>t-s) + \E_{(X_s, B_s)}\left[\varphi(s+g_0^{(t-s)})1_{\{g_0^{(t-s)}>0\}}\right]
$$
\noindent
so we obtain the a.s. convergence:
$$\lim_{t\rightarrow +\infty} \frac{\E_{(x,y)}[\varphi(g_0^{(t)}) |\F_s]}{\E_{(x,y)}[\varphi(g_0^{(t)})]} =\frac{M_s^\varphi}{\Phi(x,y)}.$$
To apply the Meta-Theorem \ref{theo:meta}, we need to prove that:
\begin{equation}\label{eq:EM}
\E_{(x,y)}[M_t^\varphi]=\Phi(x,y).
\end{equation}
Thanks to Lemma \ref{lem:g}, we have:
\begin{align*}
&\E_{(x,y)}[ \varphi(g_0^{(t)}) h(X_t,B_t)1_{\{X_t\geq  0\}}]\\
&= \varphi(0)  \E_{(x,y)}\left[  h(X_t,B_t)1_{\{X_t\geq  0\}} 1_{\{t<T_0\}}\right]  + \int_0^{t} \int_0^{+\infty} \int_\R \varphi(s) h(u,v)   \left(\int_\R |z|  p_{s}(0, z; -x, y)\overline{p}_{t-s}(0, z; u, v ) dz\right)  dv du ds\\
&=  \varphi(0)  h(x,y)1_{\{x>0\}\cup \{x=0,y>0\}}+\int_0^{t} \int_\R  \E_{(0,z)}\left[h(X_{t-s},B_{t-s}) 1_{\{T_0>t-s\}}1_{\{X_{t-s}\geq0\}}\right]  \varphi(s) p_{s}(0, z; -x, y) |z| dz ds\\
&= \varphi(0) h(x,y)1_{\{x\geq0\} }+ \int_0^{t} \int_\R  \sqrt{z^+}  \varphi(s) p_{s}(0, z; -x, y) |z| dz ds\\
&= \varphi(0) h(x,y)1_{\{x\geq0\}}+ \int_0^{t} \int_0^{+\infty}  z^{3/2}  \varphi(s) p_{s}(x,y; 0, z)  dz ds
\end{align*}
and thus:
\begin{multline}\label{eq:0t}
\E_{(x,y)}[ \varphi(g_0^{(t)}) h(X_t,B_t)1_{\{X_t\geq 0\}} +  \varphi(g_0^{(t)}) h(-X_t,-B_t)1_{\{X_t\leq0\}} ]\\
= \varphi(0) h(x,y)1_{\{x\geq0\}} + \varphi(0) h(-x,-y)1_{\{x\leq0\}} +  \int_0^{t} \int_\R  |z|^{3/2}  \varphi(s) p_{s}(x,y; 0, z) dz ds.
\end{multline}
Next:
$$\E_{(x,y)}\left[p_s(X_t,B_t;0,z)\right] = \E_{(x,y)}\left[p_{t+s-t}(X_t,B_t;0,z)\right] = p_{t+s}(x,y; 0, z)$$
hence
\begin{equation}\label{eq:tinfini}
 \int_\R |z|^{3/2} \int_0^{+\infty} \varphi(s+t) \E_{(x,y)}\big[p_s(X_t,B_t;0,z)\big] ds dz =  \int_t^{+\infty}  \varphi(s) \int_\R |z|^{3/2}  p_{s}(x,y; 0, z) dz ds \end{equation}
and the desired result (\ref{eq:EM}) follows by adding Equations (\ref{eq:0t}) and (\ref{eq:tinfini}).\\

\noindent
To prove Point $iv)$, we shall follow the ideas of Roynette, Vallois and Yor \cite{RVY3} and use enlargements of filtration. 
Let $g_0=\sup\{u\geq0;\; X_u=0\}$ and define the (progressively) enlarged filtration $(\G_t, t\geq0)$ to be the smallest filtration containing $(\F_t, t\geq0)$ and such that $g_0$ is a $(\G_t)$-stopping time. From \cite{Man}, if $W$ is a $(\Q^\varphi_{(x,y)}, (\F_t))$-Brownian motion, we have the decomposition:
$$
 W_t=W_t^{(g_0)} + \int_0^{t\wedge g_0}  \frac{d<Z^\varphi,W>_u}{Z_u^\varphi} - \int_{t\wedge g_0}^t   \frac{d<Z^\varphi,W>_u}{1-Z_u^\varphi}
 $$
 where $(W_t^{(g_0)}, t\geq0)$ is a $(\Q^\varphi_{(x,y)}, (\G_t))$-Brownian motion and $(Z_t^\varphi, t\geq0)$ denotes Az\'ema's supermartingale:
$$Z_t^\varphi=\Q^\varphi_{(x,y)}(g_0>t|\F_t).$$ 
We now compute $Z^\varphi$ in our setting. Let $\sigma_t=\inf\{s> t,\, X_s=0\}$. We have
$\{g_0>t \} = \{\sigma_t <+\infty \}$, so for $\Lambda_t\in \F_t$ :
\begin{align*}
\Q_{(x,y)}^\varphi(\Lambda_t  \cap   \{g_0>t\})&= \Q_{(x,y)}^\varphi(\Lambda_t  \cap  \{\sigma_t<+\infty\})\\
&=\lim_{n\rightarrow +\infty} \Q_{(x,y)}^\varphi(\Lambda_t  \cap  \{\sigma_t< t+n\})\\
&=\lim_{n\rightarrow +\infty} \E_{(x,y)}\left[1_{\Lambda_t \cap  \{\sigma_t< t+n\}}  M_{\sigma_t}^\varphi\right]/\Phi(x,y)\\
&= \E_{(x,y)}\left[1_{\Lambda_t }\left( \varphi(\sigma_t)h(0,|B_{\sigma_t}|)+ \int_\R |z|^{3/2} \int_0^{+\infty} \varphi(\sigma_t+s) p_s(0,B_{\sigma_t};0,z) dsdz\right)\right]/\Phi(x,y).
\end{align*}
Conditioning with respect to $\F_t$, we obtain, on the one hand :
\begin{align*}
 \E_{(x,y)}\left[\varphi(\sigma_t)h(0,|B_{\sigma_t}|)|\F_t\right] &=  \E_{(X_t,B_t)}\left[ \varphi(t+T_0)h(0,|B_{T_0}|) \right]\\
  \end{align*}
Now, from Lachal \cite{Lac1}, replacing $(X_t,B_t)$ by $(x,y)$ and using a symmetry argument :
 \begin{align*}
 & \E_{(x,y)}\left[ \varphi(t+T_0)h(0,|B_{T_0}|)1_{\{B_{T_0}\leq0\}} \right]\\
 &= \int_0^{+\infty} \int_{-\infty}^0  \varphi(t+s)h(0,|z|) |z|\left(p_s(x,y; 0, z) - \int_0^s \int_0^{+\infty}   p_{s-u}(x,y;0,-w) \Pb_{(0,-|z|)}(T_0\in du, B_{T_0}\in dw)  \right)dz ds\\
 &= \int_0^{+\infty} \int_{-\infty}^0  \varphi(t+s) |z|^{3/2}p_s(x,y; 0, z) dz ds\\
 &\qquad - \int_0^{+\infty} \int_{-\infty}^0  \varphi(t+s) |z|^{3/2} \int_0^s \int_0^{+\infty}   p_{s-u}(0,w;0,z) \Pb_{(x,y)}(T_0\in du, B_{T_0}\in dw) dz ds\\
 &= \int_0^{+\infty} \int_{-\infty}^0  \varphi(t+s) |z|^{3/2}p_s(x,y; 0, z)dz ds \\
 &\qquad - \int_0^{+\infty} \int_{-\infty}^0  \varphi(t+s) |z|^{3/2}  \E_{(x,y)}\left[1_{\{T_0<s\}}p_{s-T_0}(0,B_{T_0};0,z) \right]dz ds\\ 
  \end{align*}
while:
 \begin{align*}
 & \E_{(x,y)}\left[ \varphi(t+T_0)h(0,|B_{T_0}|) 1_{\{B_{T_0}>0\}}\right]\\
 &\qquad= \int_0^{+\infty} \int_{0}^{+\infty}  \varphi(t+s) |z|^{3/2}p_s(x,y; 0, z) dz ds\\
 &\qquad\qquad - \int_0^{+\infty} \int_{0}^{+\infty}  \varphi(t+s) |z|^{3/2}  \E_{(x,y)}\left[1_{\{T_0<s\}}p_{s-T_0}(0,B_{T_0};0,z) \right]dz ds.\\ 
  \end{align*}
On the other hand, by Fubini,
\begin{align*}
 &\int_\R |z|^{3/2} \int_0^{+\infty} \E_{(x,y)}[ \varphi(\sigma_t+u) p_u(0,B_{\sigma_t};0,z)|\F_t] du dz \\
 & \qquad= \int_\R |z|^{3/2} \int_0^{+\infty} \E_{(X_t,B_t)}[ \varphi(t+T_0+u) p_u(0,B_{T_0};0,z)] du dz\\
 & \qquad=\int_\R |z|^{3/2} \int_0^{+\infty}  \varphi(t+s)  \E_{(X_t,B_t)}[1_{\{T_0<s\}} p_{s-T_0}(0,B_{T_0};0,z)] ds dz.\\
 \end{align*}

 \noindent
Summing all the above relations gives the expression of Az\'ema's supermartingale : 
$$
Z_t^\varphi=\Q_{(x,y)}^\varphi(g_0>t|\F_t)= \frac{1}{M_t^{\varphi}} \int_\R |z|^{3/2} \int_0^{+\infty}  \varphi(t+s)  p_{s}(X_t,B_t;0,z) ds dz = \frac{N_t^\varphi}{M_t^\varphi} .
$$
Observe besides that
$$\Q_{(x,y)}^\varphi(g_0>t)=\Q_{(x,y)}^\varphi\left[Z_t^\varphi\right]=\frac{1}{\Phi(x,y)}\E_{(x,y)}\left[N_t^\varphi\right] \xrightarrow[t\rightarrow +\infty]{}0$$
since $\varphi$ has compact support.
We next set:
$$
\begin{cases}
\displaystyle m_u=  \displaystyle\varphi(g_0^{(t)}) \left(\frac{\partial }{\partial y}h(X_t,B_t)1_{\{X_t\geq 0\}}- \frac{\partial }{\partial y}h(-X_t, -B_t)1_{\{X_t\leq 0\}}\right) 
+ \int_\R |z|^{3/2} \int_0^{+\infty} \varphi(t+s) \frac{\partial }{\partial y}p_s(X_t,B_t;0,z) dsdz\\
\\
\displaystyle n_u=   \int_\R |z|^{3/2} \int_0^{+\infty} \varphi(t+s) \frac{\partial }{\partial y}p_s(X_t,B_t;0,z) dsdz
\end{cases}
$$

\noindent
Recall from Girsanov's theorem that the process $\displaystyle \left(W_t = B_t - \int_0^{t}\frac{m_u}{M_u}du, \; t\geq0\right)$ is a $(\Q^\varphi_{(x,y)}, (\F_t))$-Brownian motion. Therefore, in the filtration $(\G_t)$, we obtain the decomposition
 \begin{align*}
 W_t&= W_t^{(g_0)} + \int_0^{t\wedge g_0}\left(\frac{n_u}{N_u}-\frac{m_u}{M_u}\right)du - \int_{t\wedge g_0}^t  \frac{n_u-m_u\frac{N_u}{M_u}}{M_u-N_u} du 
 \end{align*}
 which simplifies to:
\begin{align*} 
B_t &= W_t^{(g_0)} + \int_0^{t\wedge g_0}\frac{n_u}{N_u}du - \int_{t\wedge g_0}^t  \frac{n_u-m_u }{M_u-N_u} du\\
&= W_t^{(g_0)}+   \int_0^{t\wedge g_0}\frac{n_u}{N_u}du + \int_{t\wedge g_0}^t   \frac{ \frac{\partial}{\partial y} h(X_u,B_u)1_{\{X_u\geq0\}}  - \frac{\partial}{\partial y} h(-X_u, -B_u)1_{\{X_u\leq0\}}   }{h(X_u, B_u)1_{\{X_u\geq0\}} +h(-X_u, -B_u)1_{\{X_u\leq 0\}}} du.
\end{align*}
Therefore, after time $g_0$ :
\begin{align*}
B_{t+g_0} &=  W_{t+g_0}^{(g_0)}+   \int_0^{g_0}\frac{n_u}{N_u}du + \int_{0}^t   \frac{ \frac{\partial}{\partial y} h(X_{u+g_0},B_{u+g_0})1_{\{X_{u+g_0}\geq0\}}  - \frac{\partial}{\partial y} h(-X_{u+g_0}, -B_{u+g_0})1_{\{X_{u+g_0}\leq0\}}   }{h(X_{u+g_0}, B_{u+g_0})1_{\{X_{u+g_0}\geq0\}} +h(-X_{u+g_0}, -B_{u+g_0})1_{\{X_{u+g_0}\leq 0\}}} du\\
&= B_{g_0} + \widetilde{W}_{t}^{(g_0)} + \int_{0}^t   \frac{ \frac{\partial}{\partial y} h(X_{u+g_0},B_{u+g_0})1_{\{X_{u+g_0}\geq0\}}  - \frac{\partial}{\partial y} h(-X_{u+g_0}, -B_{u+g_0})1_{\{X_{u+g_0}\leq0\}}   }{h(X_{u+g_0}, B_{u+g_0})1_{\{X_{u+g_0}\geq0\}} +h(-X_{u+g_0}, -B_{u+g_0})1_{\{X_{u+g_0}\leq 0\}}} du.
\end{align*}
where $( \widetilde{W}_{t}^{(g_0)}=W_{g_0+t}^{(g_0)}-W_{g_0}^{(g_0)}, t\geq0)$ is a Brownian motion independent from $\G_{g_0}$. Point $iv)$ finally follows from the fact that this (system of) SDEs admits a unique strong solution whose first component never reaches 0.\\
\end{proof}

\section{Penalization with the supremum}\label{sec:5}
We briefly study in this section the penalization of $X$ by a function of its supremum:
$$S_t= \sup_{u\leq t}  X_u.$$

\begin{proposition}\label{prop:Mu}
Let $\varphi:\R\longmapsto [0,+\infty[$ be a continuous function with compact support. Then the process:
$$M_u^\varphi= \varphi(S_u)h(S_u-X_u, -B_u) + \int_{S_u}^{+\infty} \varphi(z) \frac{\partial}{\partial z}h(z-X_u, -B_u) dz  $$ 
is a strictly positive and continuous martingale which converges to 0 as $t\longrightarrow +\infty$.
\end{proposition}
Observe that this martingale is a kind of analogous of Az\'ema-Yor martingale for the integrated Brownian motion. 

\begin{proof}
Assume first that $\varphi$ is differentiable. Since $h$ is harmonic for $\G$, It\^o's formula implies that $M$ is a positive and continuous local martingale. Now, from (\ref{eq:majh}) and the estimates:
$$M_t^\varphi\leq 2\| \varphi \|_{\infty} h(S_t-X_t, -B_t)\leq  2\| \varphi \|_{\infty} h\left(S_t, \sup_{u\leq t} (-B_u)\right)\leq 2\| \varphi \|_{\infty}\left(a S_t^{1/6} + b \sqrt{| \sup_{u\leq t} (-B_u)}|\right)$$
we deduce that $(M_t^\varphi, t\geq0)$ is a true martingale. The fact that $M^\varphi$ converges towards 0 is immediate since $\varphi$ has compact support. We conclude by applying the monotone class theorem to remove the assumption on the differentiability of $\varphi$.

\end{proof}

\noindent
We denote to simplify
$$\Phi(x,y)=\varphi(x)h(0,-y)+ \int_{x}^{+\infty} \varphi(z) \frac{\partial}{\partial z} h(z-x,-y) dz.$$

\begin{theorem} Let $\varphi$ be a continuous function with compact support.
\begin{enumerate}[$i)$]
\item There is the asymptotic:
$$ \E_{(x,y)}[\varphi(S_t)]\;\equi_{t\rightarrow +\infty}\;  \frac{3\Gamma(1/4)}{2^{3/4} \pi^{3/2}t^{1/4} }\Phi(x,y).$$
\item Let $u\geq0$ and $(x,y) \in  \R^2$. For any $\Lambda_u \in \F_u$, we have:
$$\lim_{t\rightarrow +\infty}\frac{\E_{(x,y)}\left[ 1_{\Lambda_u}  \varphi(S_t) \right]}{\E_{(x,y)}\left[\varphi(S_t) \right]} =\E_{(x,y)}\left[ 1_{\Lambda_u}    \frac{M_u^\varphi}{\Phi(x,y)} \right]$$
with $(M_u^\varphi, u\geq0)$ the martingale defined in Proposition \ref{prop:Mu}.
\item There exists a family of probabilities $(\Q_{(x,y)}^\varphi, \; (x,y)\in \R^2)$ such that, for any $t\geq0$:
$$\Q_{(x,y)|\F_t}^\varphi = \frac{M_t^\varphi}{\Phi(x,y)} \centerdot \Pb_{(x,y)|\F_t}.$$
\item Under $\Q_{(x,y)}^\varphi$, the r.v. $S_\infty$ is finite and its law is given by:
 $$\Q_{(x,y)}^\varphi(S_\infty \in dz) = \frac{\varphi(x) h(0,-y)}{\Phi(x,y)} \delta_x (dz) + \varphi(z)\frac{\frac{\partial h}{\partial z}(z-x,-y)}{\Phi(x,y)}1_{\{z\geq x\}}dz.$$ 
\end{enumerate}
\end{theorem}
\begin{proof}
To prove Point $i)$, we  write :
\begin{align*}
\E_{(x,y)}\left[\varphi(S_t)\right]&=\E_{(x,y)}\left[\varphi(S_t)1_{\{S_t>x\}}\right]+\E_{(x,y)}\left[\varphi(S_t)1_{\{S_t=x\}}\right]\\
&=\int_x^{+\infty} \varphi(z) \Pb_{(x,y)}(S_t\in  dz)  + \varphi(x) \Pb_{(x,y)}(S_t=x) \\
&=\int_x^{+\infty} \varphi(z) \frac{\partial}{\partial z}\Pb_{(x,y)}(T_z>t) dz+ \varphi(x)  \Pb_{(0,y)}(T_0>t)1_{\{y<0\}} \\
&=\int_x^{+\infty} \varphi(z) \frac{\partial}{\partial z}\Pb_{(z-x,-y)}(T_0>t) dz+ \varphi(x)  \Pb_{(0,y)}(T_0>t)1_{\{y<0\}} \\ 
&\equi_{t\rightarrow +\infty}\; \frac{3\Gamma(1/4)}{2^{3/4} \pi^{3/2}t^{1/4} }\left(\int_x^{+\infty}\varphi(z)\frac{\partial}{\partial z} h(z-x,-y) dz + \varphi(x)h(0,-y)\right).
\end{align*}
Then, for $t > u$, from the Markov property and since $S_u\geq X_u$:
\begin{align*}
\E_{(x,y)}\left[\varphi(S_t)|\F_u\right]&=\E_{(x,y)}\left[\varphi(S_u  \vee   \sup_{u\leq s\leq t} X_s )|\F_u\right]\\
&= \widehat{\E}_{(X_u,B_u)}\left[\varphi(S_u  \vee  \widehat{S}_{t-u})\right]\\
&\equi_{t\rightarrow +\infty} \frac{3\Gamma(1/4)}{2^{3/4} \pi^{3/2}  t^{1/4}} \left( \int_{X_u}^{+\infty} \varphi(S_u \vee z)  \frac{\partial}{\partial z} h(z-X_u,-B_u) dz  +\varphi(S_u \vee X_u)h(0,-B_u)  \right)\\
&\equi_{t\rightarrow +\infty} \frac{3\Gamma(1/4)}{2^{3/4} \pi^{3/2}  t^{1/4}} \left(  \varphi(S_u)h(S_u-X_u, -B_u)+   \int_{S_u}^{+\infty} \varphi(z)  \frac{\partial}{\partial z} h(z-X_u,-B_u) dz  \right)
\end{align*}
Therefore, Points $ii)$ and $iii)$ follow from Theorem \ref{theo:meta}.\\
To compute the law of $S_\infty$ under $\Q_{(x,y)}^\varphi$, observe that, for $c>x$:
\begin{align*}
\Q_{(x,y)}^\varphi(S_t>c)=\Q_{(x,y)}^\varphi(T_c<t)&=\E_{(x,y)}\left[   \int_{c}^{+\infty} \varphi(z)\frac{\partial}{\partial z} h(z-c,-B_{T_c}) dz  \; 1_{\{T_c<t\}}  \right]\\
&\xrightarrow[t\rightarrow +\infty]{}  \E_{(x,y)}\left[ \int_{c}^{+\infty} \varphi(z) \frac{\partial}{\partial z}  h(z-c,-B_{T_c})\right] dz
\end{align*}
from the monotone convergence theorem. Then, applying Fubini and exchanging the derivative and the expectation :
$$\Q_{(x,y)}^\varphi(S_\infty>c)= \int_{c}^{+\infty} \varphi(z) \frac{\partial}{\partial z} \E_{(x,y)}\left[  h(z-c,-B_{T_c})\right] dz$$
Observe now that, by symmetry and translation:
$$\E_{(x,y)}\left[h(z-c,-B_{T_c})\right]=\E_{(z-x,-y)}\left[h(z-c,B_{T_{z-c}})\right]$$
hence, from Lemma \ref{lem:hBTa},
$$\E_{(x,y)}\left[h(z-c,-B_{T_c})\right]=h(z-x, -y)-h(c-x, -y)$$
which proves Point $iv)$.\\
\end{proof}

\section{Appendix:  Proof of Proposition \ref{prop:T0}}\label{sec:6}

We prove in this section the following asymptotic formula (Proposition \ref{prop:T0}) for the survival function of the $n^\text{th}$ passage time at level 0: 
$$ \Pb_{(0,b)}(T_0^{(n)}>t)\; \equi_{t\rightarrow +\infty}\;    \frac{ 2^{1/4}  \Gamma(1/4) \sqrt{|b|}}{\sqrt{\pi}(n-1)!}
 \left(\frac{9}{4\pi^2}\right)^{n/2}  \frac{\left(\ln(t)\right)^{n-1} }{ t^{1/4}} $$
where $n\geq 1$ and $b\neq 0$.

\begin{proof}
Observe first that for $n=1$, Proposition \ref{prop:T0} agrees with Theorem \ref{theo:GJW}, so we now assume that $n\geq2$. Suppose that $b>0$ for simplicity. From Lachal \cite[Theorem 1]{Lac}, we have:
\begin{equation}\label{eq:LachalTn}
\Pb_{(0,b)}\left(T_0^{(n)}\in dt,\; \frac{|B_{T_0^{(n)}}|}{\sqrt{t}}\in dz\right)/(dt\,dz)
=\frac{1}{\pi^2 b t}  e^{-\frac{2}{t}b^2 -2 z^2} \int_0^{+\infty}  K_{i\gamma}\left(\frac{4bz}{\sqrt{t}}\right) \frac{\gamma \sinh(\pi \gamma)}{\left(2\cosh(\frac{\pi \gamma}{3})\right)^n}d\gamma
\end{equation}
where $K_\nu$ denotes MacDonald function with index $\nu$, see \cite[p.374]{AS}. Now, the integral may be decomposed in:
\begin{align*}
& \int_0^{+\infty}  K_{i\gamma}\left(\frac{4bz}{\sqrt{t}}\right) \frac{\gamma \sinh(\pi \gamma)}{\left(2\cosh(\frac{\pi \gamma}{3})\right)^n}d\gamma\\
&=\frac{1}{2^n}  \int_0^{+\infty}\gamma  K_{i\gamma}\left(\frac{4bz}{\sqrt{t}}\right) \frac{\sinh^3( \frac{\pi \gamma}{3}) + 3  \cosh^2( \frac{\pi \gamma}{3})\sinh( \frac{\pi \gamma}{3})}{\left(\cosh(\frac{\pi \gamma}{3})\right)^n}d\gamma\\
&= \frac{1}{2^{n-2}} \int_0^{+\infty}\gamma  K_{i\gamma}\left(\frac{4bz}{\sqrt{t}}\right) \frac{ \sinh( \frac{\pi \gamma}{3})}{\left(\cosh(\frac{\pi \gamma}{3})\right)^{n-2}}d\gamma - \frac{1}{2^n} \int_0^{+\infty} \gamma K_{i\gamma}\left(\frac{4bz}{\sqrt{t}}\right) \frac{ \sinh( \frac{\pi \gamma}{3})}{\left(\cosh(\frac{\pi \gamma}{3})\right)^{n}}d\gamma
\end{align*}
so we need to estimate:
$$ \int_0^{+\infty} \gamma K_{i\gamma}\left(\frac{4bz}{\sqrt{t}}\right) \frac{ \sinh( \frac{\pi \gamma}{3})}{\left(\cosh(\frac{\pi \gamma}{3})\right)^{k}}d\gamma\qquad \qquad (k\in \N).$$
Assume, for the moment, that the following asymptotic holds:
\begin{lemma}\label{lem:intK} For $k\geq 0$:
$$\int_0^{+\infty} \gamma K_{i\gamma}\left(\frac{4bz}{\sqrt{t}}\right) \frac{ \sinh( \frac{\pi \gamma}{3})}{\left(\cosh(\frac{\pi \gamma}{3})\right)^{k}}d\gamma\mathop{=}\limits_{t\rightarrow +\infty}   \frac{\alpha_k}{\sqrt{t}} - \beta_k \frac{(\ln(t))^{k-1}}{t^{3/4}}+ \text{\emph{o}}\left( \frac{(\ln(t))^{k-1}}{t^{3/4}}\right)$$
with
$$ \beta_0=0 \qquad \text{and,\; for }k\geq 1,\quad \beta_k =   \frac{9\sqrt{2} }{\sqrt{\pi} (k-1)!} \left(\frac{9}{4\pi^2}\right)^{k/2-1} z^{3/2}\, b\sqrt{b}. $$
\end{lemma}
\noindent
Going back to (\ref{eq:LachalTn}), we obtain :
\begin{multline*}
 \Pb_{(0,b)}\left(T_0^{(n)}\in dt,\; \frac{|B_{T_0^{(n)}}|}{\sqrt{t}}\in dz\right)/ (dz dt)  \\
 \mathop{=}\limits_{t\rightarrow +\infty} 
 \frac{e^{-2 z^2}}{\pi^2 b t}    \left( \frac{1}{\sqrt{t}}\left(\frac{\alpha_{n-2}}{2^{n-2}} -\frac{\alpha_{n}}{2^{n}}\right)+\beta_n \frac{(\ln(t))^{n-1}}{t^{3/4}}\right) + \text{o}\left( \frac{(\ln(t))^{n-1}}{t^{3/4}}\right).
 \end{multline*}
 But, since:
 $$\int_0^{+\infty} \int_0^{+\infty} \sqrt{t}  \Pb_{(0,b)}\left(T_0^{(n)}\in dt,\; \frac{|B_{T_0^{(n)}}|}{\sqrt{t}}\in dz\right) =1$$
 we must have $\displaystyle \frac{\alpha_{n-2}}{2^{n-2}} -\frac{\alpha_{n}}{2^{n}}=0$ for this function to be integrable with respect to $t$, so it remains
$$\Pb_{(x,y)}\left(T_0^{(n)}\in dt,\; \frac{|B_{T_0^{(n)}}|}{\sqrt{t}}\in dz\right)/(dtdz)\; \equi_{t\rightarrow +\infty} \;
\frac{4\sqrt{2b}}{\sqrt{\pi} (n-1)!} \left(\frac{9}{4\pi^2}\right)^{n/2}\frac{\left(\ln(t)\right)^{n-1}}{t^{7/4}} e^{-2 z^2}  z^{3/2} $$
and,   finally, integrating with respect to $z$:
\begin{align*}
\Pb_{(0,b)}(T_0^{(n)}\in dt)/dt  &\; \equi_{t\rightarrow +\infty} \;
\frac{4\sqrt{2b}}{\sqrt{\pi} (n-1)!} \left(\frac{9}{4\pi^2}\right)^{n/2}   \frac{\left(\ln(t)\right)^{n-1}}{ t^{5/4}}\int_0^{+\infty} e^{-2 z^2}  z^{3/2} dz,\\
& \;\equi_{t\rightarrow +\infty}  \;
\frac{4\sqrt{2b}}{\sqrt{\pi}(n-1)!} \left(\frac{9}{4\pi^2}\right)^{n/2}  \frac{\left(\ln(t)\right)^{n-1}}{ t^{5/4}} \frac{ \Gamma(1/4)}{2^{4+1/4}}  
\end{align*}
and the result follows by integration with respect to $t$.\\

\noindent
Therefore, it only remains to prove Lemma \ref{lem:intK}.
We shall distinguish between three cases. In the following, to simplify the notation, we set: 
$$a=\frac{4bz}{\sqrt{t}}.$$

\noindent
$\bullet$ First, assume that $k\geq 2$, so that all the integrals we are going to write are absolutely convergent. From the integral expression
$$K_{i\gamma}(a) = \int_0^{+\infty} e^{-a \cosh(t)} \cos(\gamma t) dt= a\int_0^{+\infty}  \sinh(u)   e^{-a \cosh(u)}     \frac{\sin(\gamma u)}{\gamma}  du   $$
we obtain, applying Fubini and integrating by parts with respect to $\gamma$:
\begin{align*}
\int_0^{+\infty} \gamma K_{i\gamma}\left(a\right) \frac{ \sinh( \frac{\pi \gamma}{3})}{\left(\cosh(\frac{\pi \gamma}{3})\right)^{k}}d\gamma&= \int_0^{+\infty} a \sinh(u)   e^{-a \cosh(u)}    \int_0^{+\infty} \sin(\gamma u) \frac{ \sinh( \frac{\pi \gamma}{3})}{\left(\cosh(\frac{\pi \gamma}{3})\right)^{k}}d\gamma\, du\\
& =\frac{3 a }{\pi (k-1)} \int_0^{+\infty} u\sinh(u)   e^{-a \cosh(u)}  
 \int_0^{+\infty}  \frac{\cos(\gamma u)}{\left(\cosh(\frac{\pi \gamma}{3})\right)^{k-1}} d\gamma\,du
\end{align*}
and this last expression is a cosinus transform which may be found, for instance, in Erdelyi \cite[p. 30]{Erd}\footnote{By convention, we set $\prod_{r=1}^{0} =1$. }. 
\begin{enumerate}
\item if $k=2p$ with $p\geq1$:
$$ \int_0^{+\infty}  \frac{\cos(\gamma u)}{\left(\cosh(\frac{\pi \gamma}{3})\right)^{2p-1}} d\gamma
=\frac{2^{2p-3}}{(2(p-1))!} \frac{3}{\cosh(\frac{3}{2}u)} \prod_{r=1}^{p-1}  \left(\frac{9 u^2}{4\pi^2}  + \left(r-\frac{1}{2}\right)^2  \right), $$
\item if $k=2p+1$, with $p \geq 1$:
$$ \int_0^{+\infty}  \frac{\cos(\gamma u)}{\left(\cosh(\frac{\pi \gamma}{3})\right)^{2p}} d\gamma
=\frac{4^{p-1} }{2  \pi(2p-1)!}  \frac{9 u  }{\sinh(\frac{3}{2} u)}    \prod_{r=1}^{p-1}  \left(\frac{9 u^2}{4\pi^2}  + r^2 \right). $$
\end{enumerate}
To continue the computation, we set:
$$ \Gamma_{k-1}(u)= \frac{3 u }{\pi (k-1)} \int_0^{+\infty}  \frac{\cos(\gamma u)}{\left(\cosh(\frac{\pi \gamma}{3})\right)^{k-1}} d\gamma$$
\noindent
Then, the change of variable $ v= \cosh(u)-1$ leads to:
$$\int_0^{+\infty} \gamma K_{i\gamma}\left(a\right) \frac{ \sinh( \frac{\pi \gamma}{3})}{\left(\cosh(\frac{\pi \gamma}{3})\right)^{k}}d\gamma=  ae^{-a}  \int_0^{+\infty}   e^{-a v } \Gamma_{k-1}( \text{Argcosh}(1+v ))dv$$
and another integration by parts yields:
\begin{multline*}
\int_0^{+\infty} \gamma K_{i\gamma}\left(a\right) \frac{ \sinh( \frac{\pi \gamma}{3})}{\left(\cosh(\frac{\pi \gamma}{3})\right)^{k}}d\gamma
=   ae^{-a} C_{k-1}- a^2 e^{-a} \int_0^{+\infty}   e^{-a v }\int_v^{+\infty} \Gamma_{k-1}( \text{Argcosh}(1+s ))ds\, dv
\end{multline*}
where
$$C_{k-1} =\int_0^{+\infty}\Gamma_{k-1}( \text{Argcosh}(1+s ))ds $$

\noindent
To obtain the asymptotic of the last integral, we shall apply the Tauberian theorem. Letting $v$ tend toward $+\infty$, we obtain the following asymptotic:
\begin{enumerate}
\item if $k=2p$ with $p\geq1$:
$$\Gamma_{2p-1}( \text{Argcosh}(1+v )) \equi_{v\rightarrow +\infty}    \frac{3}{\pi (2p-1)}   \frac{2^{2p-3}}{(2(p-1))!}  \frac{3}{(2v)^{3/2}}  \left(\frac{9}{4\pi^2}\right)^{p-1}  \left(\ln(v)\right)^{2p-1}$$
\item if $k=2p+1$, with $p \geq 1$:
$$\Gamma_{2p}( \text{Argcosh}(1+v )) \equi_{v\rightarrow +\infty}   \frac{3}{ \pi(2p)} \frac{4^{p-1} }{2  \pi(2p-1)!}     \frac{9}{(2v)^{3/2}}    \left(\frac{9}{4\pi^2}\right)^{p-1}  \left(\ln(v)\right)^{2p}$$
\end{enumerate}
and these two formulae reduce to:
$$\Gamma_{k-1}(\text{Argcosh}(1+v )) \equi_{v\rightarrow +\infty}    \frac{9}{\pi (k-1)!}  \frac{2^{k-3}}{(2v)^{3/2}}    \left(\frac{9}{4\pi^2}\right)^{k/2-1}  \left(\ln(v)\right)^{k-1}, $$
hence, integrating
$$\int_v^{+\infty}  \Gamma_{k-1}(\text{Argcosh}(1+s )) ds \equi_{v\rightarrow +\infty}  \frac{9}{\pi (k-1)!}  \frac{2^{k-3}}{\sqrt{2v}}\left(\frac{9}{4\pi^2}\right)^{k/2-1}\left(\ln(v)\right)^{k-1},$$
and the Tauberian theorem gives:
$$\int_0^{+\infty}   e^{-a v }\int_v^{+\infty} \Gamma_{k-1}( \text{Argcosh}(1+s ))ds \;\equi_{a\rightarrow 0} \;
 \Gamma(1/2) \frac{9}{\pi (k-1)!}  \frac{2^{k-3}}{\sqrt{2a}} \left(\frac{9}{4\pi^2}\right)^{k/2-1}\left(-\ln(a)\right)^{k-1}.$$
 Therefore, we deduce that:
 \begin{multline*}
  \int_0^{+\infty} \gamma K_{i\gamma}\left(a\right) \frac{ \sinh( \frac{\pi \gamma}{3})}{\left(\cosh(\frac{\pi \gamma}{3})\right)^{k}}d\gamma\\
  \mathop{=}\limits_{a\rightarrow 0} aC_{k-1} -  \frac{9  a^2}{\sqrt{\pi} (k-1)!}  \frac{2^{k-3}}{\sqrt{2 a}} \left(\frac{9}{4\pi^2}\right)^{k/2-1}\left(-\ln(a)\right)^{k-1} + \text{o}\left( a^{3/2}(-\ln(a))^{k-1}\right).
\end{multline*}
which is the expression given in Lemma \ref{lem:intK} after replacing  $a$ by $\displaystyle \frac{4bz}{\sqrt{t}}$.\\

\noindent
$\bullet$ When $k=1$, we rather write, integrating by parts :
\begin{align*}
\int_0^{+\infty} \gamma K_{i\gamma}\left(a\right) \frac{ \sinh( \frac{\pi \gamma}{3})}{\cosh(\frac{\pi \gamma}{3})}d\gamma &= \int_0^{+\infty}  \left(a\int_0^{+\infty}  \sinh(u)   e^{-a \cosh(u)}     \sin(\gamma u)  du  \right) \tanh( \frac{\pi \gamma}{3})d\gamma\\
& =\frac{\pi a}{3}  \int_0^{+\infty} \frac{\sinh(u)}{u}   e^{-a \cosh(u)}  
 \int_0^{+\infty}  \frac{\cos(\gamma u)}{\left(\cosh(\frac{\pi \gamma}{3})\right)^{2}} d\gamma\,du\\
 &=\frac{\pi a}{3}  \int_0^{+\infty} \frac{\sinh(u)}{u}   e^{-a \cosh(u)} \frac{9}{2\pi}  \frac{u  }{\sinh(\frac{3}{2} u)} du\\
  &=\frac{3a}{2} e^{-a} \int_0^{+\infty}e^{-a v } \frac{1}{\sinh(\frac{3}{2} \text{Argcosh}(1+v ))} dv
 \end{align*}
 and, as before, we obtain the asymptotic:
 \begin{align*}
 \int_0^{+\infty} \gamma K_{i\gamma}\left(a\right) \frac{ \sinh( \frac{\pi \gamma}{3})}{\cosh(\frac{\pi \gamma}{3})}d\gamma &=
 \frac{3a}{2} e^{-a}C_0 - \frac{3a^2}{2} e^{-a} \int_0^{+\infty}e^{-a v }  \left(\int_v^{+\infty} \frac{1}{\sinh(\frac{3}{2} \text{Argcosh}(1+s ))} ds \right)dv\\
 & \mathop{=}_{a\rightarrow 0} \frac{3a}{2}   C_0 - \frac{3 \sqrt{\pi}}{2\sqrt{2}} a^{3/2}  + o(a^{3/2})
\end{align*}
where the constant $C_0$ is given by $\displaystyle C_0=  \int_0^{+\infty}\frac{1}{\sinh(\frac{3}{2} \text{Argcosh}(1+v ))} dv$.\\

\noindent
$\bullet$ When $k=0$, we shall compute explicitly the expression:
\begin{align*}
f(a):=\int_0^{+\infty} \gamma K_{i\gamma}\left(a\right)  \sinh( \frac{\pi \gamma}{3})d\gamma=  \int_0^{+\infty}  \frac{\pi^2\sinh( \frac{\pi \gamma}{3})}{2 \sinh(\pi \gamma)} K_{i\gamma}\left(a\right)    \frac{2}{\pi^2} \gamma \sinh(\pi \gamma)  d\gamma.
\end{align*}
Indeed, applying the Lebedev transform pair:
$$\begin{cases}
\displaystyle \widehat{f}(\gamma)=\int_0^{+\infty} f(a)K_{i\gamma}(a) \frac{da}{a}\\
\displaystyle f(a) = \int_0^{+\infty} \widehat{f}(\gamma) K_{i\gamma}(a) \frac{2}{\pi^2} \gamma \sinh(\pi \gamma)  d\gamma
\end{cases}
$$
we deduce that:
\begin{align*}
 \frac{\pi^2}{2}\frac{\sinh( \frac{\pi \gamma}{3})}{ \sinh(\pi \gamma)} &= \int_0^{+\infty} f(a)K_{i\gamma}(a) \frac{da}{a}\\
 &= \int_0^{+\infty} \cos(\gamma u) \left(\int_0^{+\infty} e^{-a\cosh(u)}  f(a)\frac{da}{a}    \right)
\end{align*}
hence, from the injectiveness of the cosine transform, see Erdelyi \cite[p.30 Formula (6)]{Erd}:
$$
\frac{\pi}{2\sin(\pi/3)} \frac{1}{\cosh(u)+\cos(\pi/3)} = \int_0^{+\infty} e^{-a\cosh(u)}  f(a)\frac{da}{a}    
$$
and the injectiveness of the Laplace transform finally gives:
$$f(a)=\frac{\pi a}{\sqrt{3}} e^{-\frac{a}{2}}.$$
\end{proof}

\addcontentsline{toc}{section}{References}

\end{document}